\documentclass[]{amsart}
\usepackage{amsfonts,latexsym,rawfonts,amsmath,amssymb, amsthm,bm}
\usepackage[T1]{fontenc}
\usepackage{latexsym,lscape,rawfonts}
\usepackage{CJK}
\usepackage{mathrsfs, enumitem}
\usepackage[title]{appendix}
\usepackage{rotating}
\usepackage{graphicx}
\usepackage{galois}
\usepackage[all,cmtip]{xy}
\usepackage{extarrows,color}
\usepackage{times}

\newtheorem{thm}{Theorem}[section]

\newtheorem*{fixed point criterion}{\fixed point criterion}
\newtheorem{cor}[thm]{Corollary}
\newtheorem{lem}[thm]{Lemma}
\newtheorem{prop}[thm]{Proposition}

\theoremstyle{definition}
\newtheorem{defn}[thm]{Definition}

\newtheorem{ques}[thm]{Question}
\theoremstyle{remark}

\numberwithin{equation}{section}

\newcommand{\Z}{\mathbb Z}
\newcommand{\Q}{\mathbb Q}
\newcommand{\N}{\mathbb N}

\newcommand{\rk}{\mathrm{rk}}

\makeatletter
\newcommand*\bigcdot{\mathpalette\bigcdot@{1.5}}
\newcommand*\bigcdot@[2]{\mathbin{\vcenter{\hbox{\scalebox{#2}{$\m@th#1\bullet$}}}}}
\makeatother

\begin{document}

\title{Howson groups which are not strongly Howson}

\author{Qiang Zhang}
\address{School of Mathematics and Statistics, Xi'an Jiaotong University, Xi'an 710049, CHINA}
\email{zhangq.math@mail.xjtu.edu.cn\\ORCID: 0000-0001-6332-5476}

\author{Dongxiao Zhao}
	\address{School of Mathematics and Statistics, Xi'an Jiaotong University, Xi'an 710049, CHINA}
	\email{zdxmath@stu.xjtu.edu.cn\\ORCID: 0009-0006-7919-9244}

\thanks{The first author is partially supported by NSFC (No. 12471066) and the Shaanxi Fundamental Science Research Project for Mathematics and Physics (No. 23JSY027).}



\date{\today}

\begin{abstract}
A group $G$ is called a Howson group if the intersection $H\cap K$ of any two finitely generated subgroups $H, K<G$ is again finitely generated, and called a strongly Howson group when a uniform bound for the rank of $H\cap K$ can be obtained from the ranks of $H$ and $K$. Clearly, every strongly Howson group is a Howson group, but it is unclear in the literature whether the converse is true.  In this note, we show that the converse is not true by constructing the first Howson groups which are not strongly Howson.
\end{abstract}
\maketitle

\section{Introduction}

In 1954,  Howson \cite{How54} showed that the intersection of finitely generated subgroups of a free group is also finitely generated, which led to the notion of Howson group: a group $G$ is called a \textit{Howson group}, if the intersection $H\cap K$ of any two finitely generated subgroups $H, K<G$ is again finitely generated. (Very recently, another similar notion was introduced: if in addition, one of the two subgroups $H, K$ is normal in $G$, then the group $G$ is called \textit{weakly Howson}, see \cite{LMZ24}.) Many types of groups are Howson groups, for instance, free groups, surface groups, Baumslag-Solitar groups $BS_{1,n}=\langle a,t \mid tat^{-1}=a^n \rangle$ \cite{Mol68}, limit groups  \cite{Dah03}, etc. In particular, if $G$ is a free or surface group, then
$$\rk(H\cap K)-1\leq (\rk(H)-1)(\rk(K)-1),$$
which was conjectured by Hanna Neumann in 1957, and proved independently by Friedman \cite{Fri14} and by Mineyev \cite{Min12} for free groups in 2011, and by Antol\'{i}n and Jaikin-Zapirain \cite{AJZ22} for surface groups in 2022. Note that a free-by-(infinite cyclic group) $F_n\rtimes \Z$ with rank $n\geq 2$ is never a Howson group \cite{BB79}. Moreover, Kapovich \cite{Kap97} showed many hyperbolic groups are not Howson groups.

It is easy to show that the class of Howson groups is closed under taking subgroups and under finite extension. So in particular, virtually free groups and virtually surface groups both are Howson groups. More generally, the class of Howson groups is closed under graphs of groups, where the edge groups are finite \cite{Sky05}.
Moreover, Shusterman and Zalesskii \cite{SZ20} extended the Howson property to Demushkin groups: the intersection of a pair of closed topologically finitely generated subgroups of a Demushkin group is again topologically finitely generated. This is the first example of the Howson pro-$p$ groups that are not free.

In 2015, Ara\'ujo, Silva and Sykiotis \cite{Sky15} introduced the notion of strongly Howson: a group $G$ is called a \textit{strongly Howson} group if there exists a uniform bound for the rank of $H\cap K$ depending on the ranks of $H$ and $K$. More precisely,

\begin{defn}
  A group $G$ is \textit{strongly Howson} if
$$\xi_G(h, k):=\sup \{\rk(H \cap K) \mid H, K < G, \rk(H) \leq h, \rk(K) \leq k\}<\infty$$
for all $h,k \in \mathbb{N}$.
\end{defn}

Moreover, they showed an explicit bound for $\xi_G(h, k)$ as follows.

\begin{thm}(\cite{Sky15}, Theorem 3.2)\label{thm KP}
Let $G$ be a finite extension of a strongly Howson group $F$ and let $m=[G:F]$. Then $G$ is strongly Howson and
$$\xi_G(h,k) \leq \xi_F(m(h-1)+1,m(k-1)+1)+m-1$$
for all $h,k\geq 1$.
\end{thm}

Clearly, every strongly Howson group is a Howson group, but it is unclear in the literature whether the converse is true.
In this note, we show that the strongly Howson property is indeed stronger than the Howson property.

\begin{thm}\label{main thm}
For any $n>1$, the following groups are Howson groups, but are not strongly Howson:
$$F_n\times (\bigoplus_{m=2}^{\infty}\Z/m\Z),\quad F_n\times \Q/\Z.$$
\end{thm}

Throughout this note, let $F_n$ denote the free group of rank $n$, and let $\rk(G)$ denote the rank (i.e., the minimal number of generators) of $G$. All the abelian groups $\Z, \Q, \Z/m\Z$ and $\Q/\Z$ are viewed as additive groups.

\section{Proof of Theorem \ref{main thm}}

To prove the main theorem and make our note as self-contained as possible, we review some basic results first.

\begin{lem}[Schreier's formula]
Let $H$ be a subgroup of the free group $F_n$ with index $[F_n : H]=\ell$. Then $H$ is free with rank
$$\rk(H)=\ell (n-1)+1.$$
\end{lem}

\begin{lem}\label{finite extension lem}
The class of Howson groups is closed under taking subgroups and under finite extension.
\end{lem}

The proof of the above lemma can be an elementary exercise. Moreover, we have:

\begin{lem}
The group $F_n\times \Z (n\geq 2)$ is not a Howson group.
\end{lem}

\begin{proof}
Let $G=\langle a, b\rangle \times \langle t\rangle\cong F_2\times \Z$, and let $H=F_2=\langle a, b\rangle$, $K=\langle a, bt\rangle$ be two subgroups of $G$. Now we consider the epimorphism $$\phi: F_2\twoheadrightarrow \Z, \quad w \mapsto |w|_b,$$
where $|w|_b$ is the total $b$-exponent of $w$. Note that an element $w\in H\cap K$ if and only if $|w|_b=0$, Hence we have
$$H\cap K=\ker\phi=\langle b^{-i}ab^i\mid i\in \Z \rangle,$$ which is not finitely generated. So $F_2\times \Z$ and hence $F_n\times \Z (n\geq 2)$ is not a Howson group.
\end{proof}

In fact, if we pick $G_m=\langle a, b\rangle \times \langle t\mid t^m=1 \rangle\cong F_2\times \Z/m\Z$, and
$$H=F_2=\langle a, b\rangle<G_m, \quad K=\langle a, bt\rangle<G_m,$$
$$\phi: F_2\twoheadrightarrow \Z/m\Z, \quad w\mapsto |w|_b,$$ as above. Then by the same arguments, we have $H\cap K=\ker\phi$ which is a subgroup of $F_2$ with index $m$. Therefore, by Schreier's formula,
$$\rk(H\cap K)=m+1=m(\rk(H)-1)(\rk(K)-1)+1.$$

More generally, inspired by the above arguments and Klyachko and Ponfilenko's paper \cite{KP20}, we have:

\begin{prop}\label{ununiform bound}
Let $G=F_n\times (\bigoplus_{m=2}^{\infty}\Z/m\Z) (n>1)$ or $G=F_n\times \Q/\Z (n>1)$. Then, for any $h,k,\ell >1$, the group $G$ contains two finitely generated subgroups $H,K$ with
$$\rk(H)=h, \quad \rk(K)=k,$$
$$\rk(H\cap K)-1=\ell(\rk(H)-1)(\rk(K)-1)=\ell(h-1)(k-1).$$
Therefore, we have $\xi_G(h,k)=\infty$ for all $h,k>1$, and hence $G$ is not strongly Howson.
\end{prop}

\begin{proof}
Note that in either case, the group $G$ always contains subgroups  $G_\ell = F_2 \times \Z/\ell\Z$ for all $\ell>1$. Let $$\phi: F_2\to \Z\oplus \Z, \quad x \mapsto (\phi_1(x),\phi_2(x))$$ be the abelianization of $F_2$. We take the subgroups
$$H=\phi_1^{-1}((h-1)\Z), \ \ K_0=\phi_2^{-1}((k-1)\Z) \subset F_2,$$
$$K=\{(x, \pi(\phi_2(x)/(k-1))\mid x\in K_0\} \subset G_\ell,$$
where $\pi: \Z \to \Z/\ell\Z$ is the canonical projection. Then
$$F_2 \cap K= \phi_2^{-1}(\ell(k-1)\Z) \subset K_0,$$
$$H \cap K = H\cap (F_2\cap K)=\phi^{-1}((h-1)\Z \oplus \ell(k-1)\Z)\subset F_2.$$
Note that $H, K_0, F_2 \cap K$ and $H \cap K$ are subgroups of $F_2$ with finite index $$[F_2 : H]=h-1, \quad [F_2 : K_0]=k-1, \quad [F_2 : H\cap K]=\ell(h-1)(k-1).$$ Then by Schreier's formula, the ranks of $H, K_0, H \cap K$ are
$\rk(H)=h, \rk(K_0)=k,$ and
$$\rk(H\cap K)=\ell(h-1)(k-1)+1.$$
Moreover, note that the projection $p: K \to K_0$ is an isomorphism, we have $\rk(K)=\rk(K_0)=k$ and hence the formula
$$ \rk(H\cap K)-1=\ell(h-1)(k-1)=\ell(\rk(H)-1)(\rk(K)-1)$$
holds.
\end{proof}

\begin{prop}\label{G is Howson}
For any $n\geq 0$, all the groups $F_n\times (\bigoplus_{m=2}^{\infty}\Z/m\Z)$ and $F_n\times \Q/\Z$ are Howson groups.
\end{prop}

\begin{proof}
(1) Let $G=F_n\times (\bigoplus_{m=2}^{\infty}\Z/m\Z)$. For any given finitely generated subgroups $H, K < G$, suppose $\{x_1,\ldots, x_h\}$ and $\{y_1,\ldots, y_k\}$ are finite sets of generators of $H$ and $K$ respectively. Then there is a sufficiently large $r$ such that
$$\{x_1,\ldots, x_h, y_1,\ldots, y_k\}\subset F_n \times (\bigoplus_{m=2}^r \Z/m\Z)\subset G.$$
Therefore, $H$ and $K$ are contained in the subgroup $F_n \times (\bigoplus_{m=2}^r \Z/m\Z)$,
which is a finite extension of $F_n$ since $\bigoplus_{m=2}^r \Z/m\Z$ is finite.
Moreover, since the free group $F_n$ is a Howson group, by Lemma \ref{finite extension lem}, $F_n \times (\bigoplus_{m=2}^r \Z/m\Z)$ is a Howson group. Then the intersection $H \cap K$ is finitely generated, which means $G$ is also a Howson group.\\

(2) Let $G=F_n\times \Q/\Z$. Note that $\Q/\Z$ can be generated by $\{ 1/m \mid  2\leq m\in \N\}$. For any given finitely generated subgroups $H, K < G$, suppose $\{x_1,\ldots, x_h\}$ and $\{y_1,\ldots, y_k\}$ are finite sets of generators of $H$ and $K$ respectively, then there is a sufficiently large $r\in\N$ such that
$$\{x_1,\ldots, x_h, y_1,\ldots, y_k\}\subset F_n \times G_r\subset F_n\times \Q/\Z,$$
where $G_r=\langle 1/2, 1/3,\ldots, 1/r\rangle$ is a finite cyclic subgroup of $\Q/\Z$. It implies
$$H, ~K<F_n \times G_r.$$
Then by the same arguments as in the above case (1), we can obtain that $G=F_n\times \Q/\Z$ is also a Howson group.
\end{proof}

Finally, combining Proposition \ref{ununiform bound} and Proposition \ref{G is Howson}, we obtain Theorem \ref{main thm}.

\section{Questions and Discussions}

Recall that the semidirect product $F_2\rtimes \Z$ (in particular,  $F_2\times \Z$) is not a Howson group and hence not strongly Howson. Then, by using the same arguments as in the proof of Proposition \ref{ununiform bound}, we can obtain:

\begin{thm}\label{sufficinet condition of not Strongly Howson}
A group $G$ is not strongly Howson if it contains a subgroup isomorphic to $F_2\rtimes \Z$, or contains subgroups isomorphic to $F_2\times \Z/\ell \Z$ for infinitely many $\ell\in\N$.
\end{thm}

Moreover, inspired by Proposition \ref{G is Howson}, we can obtain a little more general result.

A group is \emph{locally finite} if every finitely generated subgroup is finite. For example, every finite group, $\bigoplus_{m=2}^{\infty}\Z/m\Z$, $\Q/\Z$ and $S_\omega$, the group of all
permutations of $\N$ which move only finitely many symbols, are all locally finite.
Let $G$ be a \emph{locally finite extension} of a Howson group $H$, that is, there is a short exact sequence as follows:
$$1\to H \xlongrightarrow{i} G\xlongrightarrow{\pi} Q\to 1,$$
where $Q$ is a locally finite group. Then, for any two finitely generated subgroups $A, B$ of $G$, let $\langle A, B\rangle \subset G$ be the subgroup generated by $A$ and $B$. Then $\pi\langle A, B\rangle$ is a finitely generated subgroup of the locally finite group $Q$, and hence, $\pi\langle A, B\rangle$ is finite. Now, let us consider the extension $G'$ of the Howson group $H$ by the finite subgroup $\pi\langle A, B\rangle$ of $Q$,
$$1\to H \xlongrightarrow{i} G'\xlongrightarrow{\pi} \pi\langle A, B\rangle\to 1.$$
Then $G'$ is again Howson by Lemma \ref{finite extension lem}, and hence $A\cap B$ is finitely generated because $A, B\subset \langle A, B\rangle \subset G'$. Therefore, $G$ is also a Howson group. Namely, we have obtained:

\begin{thm}\label{locally finite extension is Howson}
The class of Howson groups is closed under locally finite extension.
\end{thm}

Now, by combining the above two theorems, we obtain an approach to construct Howson but not strongly Howson groups as follows.

\begin{cor}
Let $H$ be a Howson group containing a subgroup isomorphic to $F_2$, and let $Q$ be a locally finite group containing subgroups isomorphic to $\Z/\ell \Z$ for infinitely many $\ell$. Then the directly product $H\times Q$ is a Howson group but not strongly Howson.
\end{cor}

For instance, for any $n\geq 2$, the groups $F_n\times (\bigoplus_{m=2}^{\infty}\Z/m\Z), F_n\times \Q/\Z$ and $F_n\times S_\omega$ are Howson but not strongly Howson. Note that these groups are not finitely generated, so we wonder:

\begin{ques}\label{Question f.g}
Is there a finitely generated or finitely presented Howson group which is not strongly Howson?
\end{ques}

To construct a positive example for Question \ref{Question f.g}, one possible approach is, first to pick a finitely generated group $B$ which does not contain $\Z$ but contains subgroups isomorphic to $\Z/\ell \Z$ for infinitely many $\ell\in\N$. (Note that such a group does exist by Golod's Theorem which answered the famous Burnside's Problem: \emph{There exists a finitely generated infinite group $B$ such that every element of $B$ has finite order}. In fact, the orders of all elements in $B$ can be not uniformly bounded, see \cite{Ale72, Gri80, Osi24}.) Then, we consider the group $G=F_n\times B$, which is finitely generated and is not strongly Howson. However, it is difficult to guarantee $G$ to be a Howson group.

Another possible approach is to embed the group  $G=F_n\times (\bigoplus_{m=2}^{\infty}\Z/m\Z)$ or $G=F_n\times \Q/\Z$ into a finitely generated Howson group $G^*$.
Since $G$ is countable, we may use the following Higman-Neumann-Neumann's Embedding Theorem \cite{HNN49} to get such a $G^*$:

\begin{thm}[Higman-Neumann-Neumann's Embedding Theorem]\label{HNN embedding}
Every countable group $G$ can be embedded in a 2-generator group $G^*$. More precisely, suppose $G=\langle c_1, c_2, \ldots\mid r_1, r_2, \ldots\rangle$, then $G^*$ can be an HNN extension of the free product $G\ast F_2$ and has a presentation
$$G^*=\langle G, a, b, t\mid t^{-1}at=b, ~t^{-1}b^{-1}abt=c_1a^{-1}ba, ~~t^{-1}b^{-2}ab^2t=c_2a^{-2}ba^2, \ldots\rangle.$$
Clearly, $G^*$ can be generated by the two generators $a$ and $t$, and $G^*$ is finitely presented if $G$ is finitely presented.
\end{thm}

It is easy to see that the 2-generator group $G^*$ is not strongly Howson. But, it seems difficult to guarantee $G^*$ to be a Howson group?

\vspace{6pt}

\noindent\textbf{Acknowledgements.} The first author would like to thank Shengkui Ye for some helpful communications.


\end{document}